\documentclass[a4paper, twoside,12pt]{article}
 \usepackage{fancyhdr}
\usepackage{fnpos}
 \usepackage[english]{babel}        
\usepackage[T1]{fontenc}          
\usepackage{graphicx}             
\usepackage{makeidx}
\usepackage{fancybox}
\usepackage{framed}
\usepackage{fancyhdr}
 \usepackage{pstricks,pst-plot,pstricks-add}
\usepackage[margin=1in]{geometry}
\usepackage{graphicx}
\usepackage{titlesec}
\usepackage{amsmath}
\usepackage{amsfonts}   
\usepackage{amssymb}    
\usepackage{amsthm}
\usepackage{dsfont}
\usepackage{fancyhdr}
\usepackage{fnpos}
 \usepackage[english]{babel}        
\usepackage[T1]{fontenc}
\usepackage{dsfont}       
\usepackage{graphicx}             
\usepackage[margin=1in]{geometry}
\usepackage{graphicx}
\usepackage{amsfonts}   
\usepackage{amsthm}
\usepackage{amsmath,amssymb,color,inputenc,euscript,graphicx,psfrag}

\newtheorem{thm}{Theorem}[section]
\newtheorem{cor}[thm]{Corollary}
\newtheorem{lem}[thm]{Lemma}

\numberwithin{equation}{section}

\renewcommand{\thefootnote}

\author {B\'echir Amri$^*$ and Khawla Kerfaf $ ^{**}$}

\title{ On the translation invariant operators in $\ell^p(\mathbb{Z}^d)$  }
\date{}
\begin{document}
\maketitle
\begin{center}
   $^*$Taibah University, College of Sciences, Department of Mathematics, P. O. BOX 30002, Al Madinah AL Munawarah, Saudi Arabia.\\
\textbf{ e-mail:} bechiramri69@gmail.com\\
  $^{**}$Universit\'{e} Tunis El Manar, Facult\'{e} des sciences de Tunis,\\ Laboratoire d'Analyse Math\'{e}matique
       et Applications,\\ LR11ES11, 2092 El Manar I, Tunisie.\\
     \textbf{ e-mail:} kh.karfaf@gmail.com
\end{center}
\begin{abstract}
  In this paper we study boundedness of translation invariant operators in  the discrete space $\ell^p(\mathbb{Z}^d)$.  In this  context a Mikhlin type multiplier theorem   is  given,  yielding boundedness for   certain known    operators . We also give
  $\ell^p-\ell^q$ boundedness of a discrete wave equation.
\end{abstract}
\section{Introduction}
It is well known that translation invariant operator from
$L^p$ into $L^q$ may be represented by     convolution with a tempered distribution, or equivalently by   Fourier multiplier
transformation. This was originally proved  in    the classical article of   H\"{o}rmander \cite{Hor}. Through  many aspects of harmonic analysis,
many studies have been devoted to the  topic of the  $L^p$-bounded of    translation invariant operator.
The  most famous are the works of    Calder\'{o}n and Zygmund on  the singular integral operators, with a large number of generalizations.\par
In this paper we consider  translation invariant operator $T$ on $\mathbb{Z}^d $.
The  problem is essentially the multiplier problem,
$$\mathcal{F}_{\mathbb{Z}^d}(T(f))=m \mathcal{F}_{\mathbb{Z}^d}(f)$$
where the function $m$   is   defined on the  Torus  $  \mathbb{R}^d/\mathbb{Z}^d$. In this setting  a  H\"{o}rmander's type  theorem  for  $\ell^p-\ell^q$ boundedness of  $T$  and  an $\ell^p$-theorem of   Mikhlin- type are  given.  We apply  our   results to get $L^p$-estimate for the discrete wave equation.
\par We begin by introducing  the following notations.
Let $\mathbb{T}^d = \mathbb{R}^d/\mathbb{Z}^d $ be the $d$-dimensional
torus. Functions on  $\mathbb{T}^d$ are functions $f$ on $\mathbb{R}^d$ that satisfy $f (x+n) = f (x)$ for all $x \in \mathbb{R}^d$
and $n\in \mathbb{Z}^d$. Such functions are called $1$-periodic in every coordinate. Haar measure
on  $\mathbb{T}^d$   is the restriction of $d$-dimensional Lebesgue measure to the set
$[0,1)^d$.  This measure is still denoted by $dx$ and given by
$$\int_{\mathbb{T}^d}f(x)dx=\int_{[0,1)^d}f(x)dx.$$
We  denote by  $L^p(\mathbb{T}^d)$, $1\leq p\leq \infty$    the  Lebesgue space $L^p([0,1)^d,dx)$.
The inner product  of the Hilbert space  $L^2(\mathbb{T}^d)$ is given by
$$\langle f,g \rangle=\int_{  \mathbb{T}^d}f(\xi) \overline{g(\xi)}d\xi.$$
The functions $\psi_n: \xi\rightarrow e^{2i\pi  \xi.n}$, indexed by $n\in \mathbb{Z}^d$, form   a complete orthonormal system of $L^2(\mathbb{T}^d)$
 where  for $x=(x_1,...,x_d)$ and  $y=(y_1,...,y_d)$ in $\mathbb{R}^d$
 $$x.y=x_1y_1+...+x_dy_d\qquad\text{and}\qquad |x|=(x.x)^{1/2}.$$
  By  $\ell^p= \ell^p(\mathbb{Z}^d)$, $1\leq p < \infty$ , we denote the usual Banach space of p-summable complex-valued function  $f=(f(n))_{n\in \mathbb{Z}^n}$
equipped with the norm
$$\|f\|_{\ell^p}=\left(\sum_{n\in \mathbb{Z}^d } |f(n)|^p\right)^{\frac{1}{p}} $$
and $\ell^\infty $  the space of bounded function on $\mathbb{Z}^d$ with $\|f\|_{\ell^\infty}=\sup_{n\in \mathbb{Z}^d}|f(n)|$.
 We note the following  elementary embedding relations
 \begin{equation}\label{llp}
   \ell^q\subset \ell^p; \quad \text{and}\quad \|f\|_{\ell^q}\leq \|f\|_{\ell^p},\quad 1\leq p\leq q\leq \infty.
\end{equation}
 \par For   $f\in \ell^2(\mathbb{Z}^d)$   its Fourier transform is given by
 $$\mathcal{F}_{\mathbb{Z}^d}(f)(\xi)=\sum _{n\in \mathbb{Z}^d}f(n)e^{i2\pi n.\xi},  \qquad \xi\in  \mathbb{R}^d.$$
The  Fourier transform $\mathcal{F}_{\mathbb{Z} ^d}$  is an isometry
from  $\ell^2(\mathbb{Z}^d)$ into $L^2(\mathbb{T}^d)$ and its  inverse $\mathcal{F}_{\mathbb{Z}^d}^{-1}$   is given by
$$\mathcal{F}_{\mathbb{Z}^d}^{-1}(u)(n)= \int_{\mathbb{T}^d} u(\xi)e^{-i2\pi  \xi.n }dx,\qquad u\in L^2(\mathbb{T}^d). $$
By Riesz-Thorin convexity theorem, the  Fourier transform $\mathcal{F}_{\mathbb{Z} ^d}$ and its     inverse  $\mathcal{F}_{\mathbb{Z}^d}^{-1} $
 satisfy  the Hausdorff-Young inequalities
 \begin{equation}\label{12}
    \|\mathcal{F}_{\mathbb{Z}^d}(f)\|_{L^{p'}} \leq    \|f\|_{\ell^p},
 \end{equation}
and
 \begin{equation}\label{13}
    \|\mathcal{F}_{\mathbb{Z}^d}^{-1}(u)\|_{\ell^{p'}} \leq    \|u\|_{L^p},
 \end{equation}
 for , $1<p\leq 2 $ and $1/p+1/p'=1$.
 \par  Convolution product of two functions $f$  and $g$  of $\ell^2(\mathbb{Z}^d)$  is defined by
 $$f*_{\mathbb{Z}^d}g(n)=g*_{\mathbb{Z}^d}f(n)=\sum_{k\in\mathbb{Z}^d}f(k)g(n-k); \quad n\in \mathbb{Z}^d.$$
 If  $f,\;g\in \ell^1(\mathbb{Z}^d)$ then $f*_{\mathbb{Z}^d} g \in \ell^1(\mathbb{Z}^d)$  and $$\mathcal{F}_{\mathbb{Z}^d}(f*_{\mathbb{Z}^d}g)=\mathcal{F}_{\mathbb{Z}^d}(f)\mathcal{F}_{\mathbb{Z}^d}(g).$$
 Suppose  $f \in \ell^p(\mathbb{Z}^d)$  and $g \in \ell^q(\mathbb{Z}^d)$ with $1 \leq  p, q, r \leq \infty$,
  $\frac {1}{p}+\frac {1}{q}=\frac {1}{r}+1 .$ Then
\begin{equation}\label{ygin}
 \|f*_{\mathbb{Z}^d}g\|_{\ell^r}\leq \|f\|_{\ell^p}\|g\|_{\ell^q} \quad (\text{Young's Inequality}).
\end{equation}
\section{Translation invariant operators }
In this section we shall be concerned with the space   of bounded operators  $T$ from $\ell^p$ to $\ell^q$ for $p\leq q$,
  which commute with translations; that is,  $\tau_n T=T\tau_n$
for all $n\in \mathbb{Z}^d$, where $\tau_n(f)(k)=f(n+k)$. It is not difficult to see that $T$ is a convolution operator. Indeed, let
$$K=T(\mathds{1}_{\{0\}}),$$
where $\mathds{1}_A$ is the  characteristic function of a set $A$. Consider first, functions  $f$ with compact support and write
 $$f=\mathds{1}_{\{0\}}*_{\mathbb{Z}^d}f= \sum_{n\in \mathbb{Z}^d}f(n)\tau_{-n}(\mathds{1}_{\{0\}}).$$
Since  $T$ is translation invariant operator  then
 $$T(f)=T(\mathds{1}_{\{0\}}*_{\mathbb{Z}^d} f)=  \sum_{n\in \mathbb{Z}^d}f(n)T(\tau_{-n}(\mathds{1}_{\{0\}}))= \sum_{n\in \mathbb{Z}^d}f(n) \tau_{-n}T(\mathds{1}_{\{0\}}))=
  K*_{\mathbb{Z}^d}f.$$
Now  for  function  $f\in \ell^1(\mathbb{Z}^d)$  we  let
$$f_j=\sum_{|n|\leq j}f(n) \mathds{1}_{\{n\}}, \qquad j \geq 0$$
Clearly  the sequence $(f_j)_j$ converges   to $f$  in   $\ell^r(\mathbb{Z}^d)$  for all $1\leq r<\infty$  and
$$ \|T(f_j)-K*_{\mathbb{Z}^d} f\|_{\ell^q}=\|K*_{\mathbb{Z}^d}(f_j-f)\|_{\ell^q}\leq \|K\|_{\ell^q}\|f_j-f\|_{\ell^1}$$
which implies that $(T(f_j))_j$ converges to $K*_{\mathbb{Z}^d} f$ in $\ell^q(\mathbb{Z}^d)$. But $T: \ell^p(\mathbb{Z}^d)\rightarrow \ell^q(\mathbb{Z}^d)$ is bounded  and $(f_j)_j$ converges   to $f$  in   $\ell^p(\mathbb{Z}^d)$,  thus  by uniqueness of the  limit we may have   $T(f)=K*_{\mathbb{Z}^d}f$.
Notice that the $\ell^p-\ell^q$ boundedness of $T$ implies  that $K\in \ell^q(\mathbb{Z}^d)$. We state the following
\begin{thm}
 If  $T$ is a bounded translation invariant operator from  $\ell^p(\mathbb{Z}^d)$ to  $\ell^q(\mathbb{Z}^d)$, $p\leq q$,  then
there is  exists a function $K\in \ell^q(\mathbb{Z}^d)$  such that
$$  T(f)=K*_{\mathbb{Z}^d}f,\qquad  f\in\ell^1(\mathbb{Z}^d).$$
\end{thm}
\par  Translation invariant operator can also   be described as Fourier multiplier transformation   $T_m$ defined by
  $$\mathcal{F}_{\mathbb{Z}^d}(T_m(f))=m\mathcal{F}_{\mathbb{Z}^d}(f) $$
 where $m$  is   a  bounded  measurable function $m$ on $\mathbb{T}^d$.
An important class of $T_m$ is given by   $ L^{r,\infty}(\mathbb{T}^d)$ for $r>1$,  that is  the space  of     measurable functions $m$  such that for  some constant $c>0$,
\begin{equation}\label{LL}
  \int_{\{\xi\in (0,1 )^d,\; |m(\xi)|\geq s\}} dx\leq \frac{c}{s^r}, \qquad s>0.
\end{equation}
In particular if  $m$   satisfies the estimate
\begin{equation}\label{mc}
 |m(\xi)|\leq c | \xi |^{ -r}; \qquad \xi\in(0,1)^d,\; \xi\neq0,
\end{equation}
   for $0<\alpha<d$ then $m\in L^{  d/r,\infty}(\mathbb{T}^d)$.
\begin{thm}\label{th1}
If   $m\in L^{\alpha,\infty}(\mathbb{T}^d)$ with  $\alpha > 1$
  then $T_m$ is a bounded operator from  $\ell^{p}(\mathbb{Z}^d)$ into $ \ell^{q}(\mathbb{Z}^d)$, provided that
$$1<p\leq2\leq q<\infty \quad,\quad \dfrac{1}{p}-\dfrac{1}{q}= \frac{1}{\alpha}.$$
 \end{thm}
  This result is originally proved in \cite{Hor} for translation invariant operator on $\mathbb{R}^d$  and  in   \cite{stein} for  translation invariant operator  on $\mathbb{Z}$, for completeness we  extended this result to  $\mathbb{Z}^d$.  The proof of Theorem \ref{th1} follows closely the  argument of  \cite{Hor}.
\begin{lem}\label{l1}
Let $\varphi\geq 0$ be a measurable function such that for some constant $c>0$
\begin{equation}\label{phi}
  \int_{\{\xi\in (0,1 )^d,\;\varphi(\xi)\geq s\}} \;dx   \leq\dfrac{c}{s},\qquad s>0.
\end{equation}
Then for all $1<p\leq  2$ there exists a constant $c_{p}>0$ such that
\begin{equation}\label{01}
   \Big(\int_{(0,1)^d} |\mathcal{F}_{\mathbb{Z}^d}(f)(\xi)|^p |\varphi(\xi)|^{2-p} d\xi \Big )^{\frac{1}{p}} \leq  c_{p}\;\| f \|_{ \ell^p };  \quad   f\in \ell^p(\mathbb{Z}^d).
\end{equation}
\end{lem}
 \begin{proof}
  Put $d\mu(\xi)=\varphi^2(\xi)\;d\xi$ and let  $T$ be the operator  defined  on $ \ell^1(\mathbb{Z}^d)$ by
  $$ T(f) =\frac{ \mathcal{F}_{\mathbb{Z}^d}(f)}{\varphi}\,.$$
 Noting   that $T(f)$ is well defined $\mu$-almost everywhere on $(0,1)^d$, since we have that  $\mu (\{\xi\in (0,1)^d,\;\; \varphi(\xi)=0\})=0$.
 In fact, for $s>0$ we have
 \begin{eqnarray*}
 \mu (\{\xi\in (0,1)^d,\;\; \varphi(\xi)\leq s\})&=& \int_{\{\xi\in (0,1)^d,\;\; \varphi(\xi)\leq s\}}\varphi(\xi)^2\,d\xi
 \\&=&2\int_{\{\xi\in (0,1)^d,\;\;  \varphi(\xi)\leq s\}}\int_{ 0\leq t\leq \varphi(\xi)  }t\,dt\,d\xi
 \\&\leq &2\int_0^s\int_{\{\xi\in (0,1)^d,\;\;   t\leq \varphi(\xi) \}}t\,dt\,d\xi
 \\&\leq &2c\;\int_0^s dt =2cs
 \end{eqnarray*}
which  implies that $\mu (\{\xi\in (0,1)^d,\;\; \varphi(\xi)=0\})\leq 2cs$,  for all $s>0$.
\par Now   for $s>0$ and $f\in \ell^1(\mathbb{Z}^d)$  we have
\begin{eqnarray*}
\mu\,(\lbrace \xi\in (0,1)^d,\;\; |T(f)(\xi)|\geq s \rbrace )&\leq   \mu\,\left( \left\{ \xi\in (0,1)^d,\;\;  \varphi(\xi)\leq \frac{\|f\|_{\ell^1}}{s} \right\}\right)  \leq   \dfrac{2c \|f \|_{\ell^1 }}{s}\,.
\end{eqnarray*}
 Hence   $T$ is of weak type $(1.1)$.  In addition, from Plancherel Theorem
$$ \mu(\lbrace\xi\in (0,1)^d,\;\; |T(f)(\xi)|\geq s \rbrace) \leq  \dfrac{\Vert f \Vert_{ \ell^2 }^2  }{s^2} \,,  $$
which mean that $T$ is of weak type (2,2). We  thus obtain  Lemma \ref{l1}  by using  Marcinkiewicz interpolation Theorem.
\end{proof}
\begin{lem}\label{2.6}
If $ \varphi $ satisfies  (\ref{phi}) and  $ 1<p<r<p^{'}<\infty $, then we have
  $$\Big( \int_{  (0,1)^d}  \Big| \mathcal{F}_{\mathbb{Z}}(f)(\xi)(\varphi(\xi))^{(1/r-1/p^{'})} \Big|^{r}d\xi \Big)^{1/r} \leq C_{p} \Vert f \Vert_{ \ell^p }; \qquad           f\in \ell^p(\mathbb{Z}^d). $$
\end{lem}
 \begin{proof}
  Put $ a=(p^{'}-p )/ (p^{'}-r)$ and $a^{'}$ it's conjugate. We note the following
  $$  \frac{p}{a}+\frac{p'}{a'}=r, \quad\left(1- \frac{r}{p'}\right)a = 2-p ,\quad  \left(r- \frac{p}{a}\right)a'= p'. $$  Then using Holder's inegality,(2.5) and the Hausdorff-Young inegality (\ref{12})
\begin{eqnarray*}
& &\Big(  \int_{(0,1)^d}  | \mathcal{F}_{\mathbb{Z}^d}(f)(\xi) |^{r} |\varphi(\xi)|^{(1-r/p^{'})} d\xi \Big)^{1/r} \\
& \leq &    \Big( \int_{(0,1)^d}  | \mathcal{F}_{\mathbb{Z}^d}(f)(\xi) |^{p} |\varphi(\xi)|^{(2-p)} d\xi \Big)^{1/ra}  \Big( \int_{(0,1)^d} | \mathcal{F}_{\mathbb{Z}^d}(f)(\xi) |^{p^{'}} d\xi \Big)^{1/ra^{'}} \\
& \leq &  c_{p} \Vert f \Vert_{ \ell^p }
\end{eqnarray*}
which is the desired statement.
\end{proof}
\begin{proof}[Proof of Theorem \ref{th1}]
Assume first that $ p \leq q^{'} $ and let $ \varphi = |m|^{ \alpha}$. Clearly from  (\ref{LL})  the function      $ \varphi $  satisfies  the condition (\ref{phi}).   Hence using    Lamma \ref{2.6} with   $ r = q^{'} $   and  the fact that
 $ 1/p - 1/q = 1 / q^{'} - 1/p^{'} = 1/\alpha, $ we obtain that
  \begin{equation}\label{f1}
  \left(\int_{(0,1)^d} |m(\xi) \mathcal{F}_{\mathbb{Z}^d} (f)(\xi)|^{q'}d\xi\right)^{1/q'} \leq c  \|f\|_{ \ell^p }.
    \end{equation}
Now the Hausdorff-Young inegality ( \ref{13} ) implies
$$\|T_m(f)\|_{ \ell^q} \leq \|m \mathcal{F} (f)\|_{q'} \leq c_{p} \|f\|_{ \ell^p }.$$
When $ q^{'} < p = (p^{'})^{'} $, we can apply the similar argument to the adjoint operator $ T^{*}_{m} = T_{\overline{m}} $, since $ 1<q^{'} \leqslant 2 \leqslant p^{'} < \infty $ and $ 1/q^{'} -1/p^{'} = 1/\alpha $. Hence by duality it follows that
$$\|T_{m}(f)\|_{\ell^q }  \leq c_{p} \|f \|_{\ell^p }.$$
This finishes the proof of Theorem \ref{th1}.
\end{proof}
 \begin{cor}
   If  $m$  satifies (\ref{mc}) with   $ 0<r<d$
then  $T_m$ is bounded from  $\ell^{p}(\mathbb{Z}^d)$ into $ \ell^{q}(\mathbb{Z}^d)$, provided that
$$1<p\leq2\leq q<\infty \quad,\quad \dfrac{1}{p}-\dfrac{1}{q}=  \frac{r}{d}.$$
 \end{cor}
Now   observe that the inequality (\ref{LL}) can be restricted  only to $s\geq1$  which implies that  $L^{\alpha,\infty}(\mathbb{T}^d)\subset L^{\beta,\infty}(\mathbb{T}^d)$  for all $1 < \beta\leq \alpha$. Thus  one  can state
\begin{cor}
If   $m\in L^{\alpha,\infty}(\mathbb{T}^d)$ with  $\alpha > 1$
  then  the operator $T_m$ is bounded from  $\ell^{p}(\mathbb{Z}^d)$ into $ \ell^{q}(\mathbb{Z}^d)$, provided that
$$1<p\leq2\leq q<\infty \quad,\quad \dfrac{1}{p}-\dfrac{1}{q}\geq \frac{1}{\alpha}.$$
 \end{cor}
 \par We now study   $L^p$-boundedness of  the  multiplier operator $T_m$. We begin by   the following:
  \begin{thm}\label{t11}
    If $m$ is a $C^{d+1} $- function  on $\mathbb{T}^d$ then $T_m$ is a bounded operator from  $\ell^{p}(\mathbb{Z}^d)$ into
     itself  for all
    $1\leq p \leq \infty$.
   \end{thm}
   \begin{proof}
      We note first that  the kernel of $T_m$ is given by
     $$K(n)=\int_{(0,1)^d}m(\xi)e^{-2i\pi\xi.n}\;d\xi, \quad n\in \mathbb{Z}^d.$$
    Using integrations by parts we have
     $$|n_1^{\gamma_1}...n_d^{\gamma_d}K(n)|= \left|
     \int_{(0,1)^d} \partial^{\gamma_1}_{\xi_1}...\partial^{\gamma_d}_{\xi_d}m (\xi)\;e^{-2i\pi\xi.n}\;d\xi\right|
     \leq c$$
   for all  $ \gamma_1,...,\gamma_d \in \mathbb{N}$ with  $ \gamma_1+...+\gamma_d\leq d+1$ .
  It follows that
  $$ |K(n)|\leq \frac{c}{(1+|n_1|)...(1+|n_j|)^2...(1+|n_d|)\;}$$
  for all $j=1,...,d$.
   By varying $j$ from $1$ to $d$ we deduce the following estimate
    $$ |K(n)|\leq \frac{c}{\Big((1+|n_1|)...(1+|n_j|) ...(1+|n_d|)\Big)^{1+1/d}\;}$$
  which implies that the kernel $K$ is in $\ell^1(\mathbb{Z}^d)$. This yields the result.
   \end{proof}
Our main result is  the following   H\"{o}rmander-Mihlin type multiplier theorem  where  we may consider   $T_m$  as a  Calder\'{o}n-Zygmund operator.
\begin{thm}\label{th2}
   Let  $m$ be a bounded function on the torus  $\mathbb{T}^d$. We assume that $m$ is $C^{d+1}$-function on $\mathbb{R}^d\setminus \mathbb{Z}^d$  and satisfies the Mikhlin condition,
   \begin{equation}\label{mik}
    |\partial^\alpha_\xi m(\xi) |\leq c|\xi|^{-|\alpha|}, \qquad  \xi\in (-1/2,1/2)^d
  \end{equation}
    for all $\alpha=(\alpha_1,...,\alpha_d)\in \mathbb{N}^d$ with $|\alpha|=\alpha_1+...+\alpha_d\leq d+1$. Then $T_m$ extended to a bounded operator from $\ell^p$ into itself for all $1<p <\infty$.
 \end{thm}
 \begin{proof}
Taking $\psi$ a  $C^\infty$-function  on $\mathbb{R}^d $ such that $\psi(\xi)=1$ for $|\xi|\leq 1/16$ and $\psi(\xi)=0$ for $|\xi|\geq 1/8$.
In $(-1/2,1/2)^d$ we split $m$ into
$$m=(1-\psi)m+m\psi=m_1+m_2.$$
Since $m_1=m$ near the sides  $|\xi_j|=1/2$, then $m_1$ can be extended to $C^{d+1}$-function  on $\mathbb{T}^d$ and by  Theorem \ref{t11} the operator $T_{m_1}$ is bounded on
$\ell^p$  for all $1\leq p\leq \infty$. It is therefore enough to  prove   boundedness of $T_{m_2}$.  Introduce   the function   $\mathcal{K}$  by
$$\mathcal{K}(x)= \int_{ (-1/2,1/2)^d}m_2(\xi)e^{-i2\pi\xi.x}d\xi, \qquad x\in \mathbb{R}^d  $$
and $K$ its restriction to $\mathbb{Z}^d$. One can write
$$T_{m_2}(f)=K*_{\mathbb{Z}^d}f, \qquad f\in \ell^1.$$
 Our aim  is   to prove that $T_{m_2}$ is a Calderon-Zygmund operator. We consider here $\mathbb{Z}^d$ as a  space of homogeneous type in the sense of Coifman and Weiss \cite{CW},    equipped  with the Euclidean metric   $ (r,s)\rightarrow |r-s| $ and the counting  measure. Precisely, we  will prove that the   kernel $K$ satisfies  the integral H\"{o}rmander condition: there exists constant $c>0$ such that for all $s\in \mathbb{Z}^d$,
 \begin{equation}\label{HC}
   \sum_{r\in \mathbb{Z}^d,\; |r|\geq 2|s|} |K(r-s)-K(r)|\leq c.
 \end{equation}
\par To begin, let  $\phi$ be    a  $C^\infty$- function on $\mathbb{R}$,  such that $supp(\phi)\subset \{t\in \mathbb{R};\: \; 1/2\leq |t|\leq 2\}$ and
$$\sum_{j\in  \mathbb{Z}}\phi(2^{-j}t)=1,\qquad t\neq0.$$
So we have
 $$\sum_{j=3}^\infty\phi(2^{-j}|\xi|)=1,\qquad |\xi|\leq \frac{1}{8}.$$
 and  we may write $$m_2(\xi)=\sum_{j=3}^\infty m_2(\xi)\phi(2^{-j}|\xi|)=\sum_{j=3}^\infty m_j (\xi), \qquad \xi\in (-1/2,1/2)^d.$$
Notice that
\begin{equation}\label{411}
  \sum_{j=3}^\infty |m_j(\xi)|\leq c
\end{equation}
for some constant $c>0$, since this sum contains at most  three non-null terms.
We now set
$$\mathcal{K}_j(x)= \int_{ (-1/2,1/2)^d}m_j(\xi)e^{-i2\pi\xi.x}d\xi, \qquad x\in \mathbb{R}^d. $$
By  Fubini's Theorem  and (\ref{411}) the sum
$\sum_j \mathcal{K}_j$ converges and
$$\sum_{j=3}^\infty \mathcal{K}_j(x)=\int_{ (-1/2,1/2)^d}\sum_{j=3}^\infty m_j(\xi)e^{-i2\pi\xi.x}d\xi=\int_{ (-1/2,1/2)^d}m_2(\xi)e^{-i2\pi\xi.x}d\xi= \mathcal{K}(x),$$
Next we shall give estimates of the kernels  $\mathcal{K}$. Observe first   that $m_2 $ satisfies  the condition $(\ref{mik})$  and  from  this  estimate    and the compactness of $supp(\varphi)$  one can    obtain the following
$$|\partial^{\alpha}m_j(\xi)|  \leq c\;2^{-j|\alpha|};\quad\text{and}\quad  |\partial^{\alpha}(\;\xi_im_j(\xi)\;)|  \leq c\;2^{-j(|\alpha|-1)};
\;\;
 $$
 for $ i=1,...,d$ and for $|\alpha|\leq d+1$. It follows that
 \begin{equation}\label{141}
  |x|^s|\mathcal{K}_j(x)| \leq c \sum_{|\alpha|=s}\|\partial^\alpha m_j\|_{L^1}\leq  c\; 2^{j(d-s)}
 \end{equation}
 and
 \begin{equation}\label{142}
  |x|^s  \left| \frac{\partial \mathcal{K}_j}{\partial x_i}(x) \right |\leq c \sum_{|\alpha|=s}\|\partial^\alpha (\xi_i m_j(\xi))\|_{L^1}\leq  c\; 2^{j(d+1-s)}.
  \end{equation}
 Using (\ref{141}) with $s=0$ and $s=d+1$ we get
$$\sum_{j=3}^\infty|\mathcal{K}_j(x)|=\sum_{2^j\leq |x|^{-1}}|\mathcal{K}_j(x)|+\sum_{2^j> |x|^{-1}}|\mathcal{K}_j(x)|\leq c |x|^{-d}.$$
Similarly by  (\ref{142})  with $s=0$ and $s=d+2$
$$\sum_{j=3}^\infty\left|\frac{\partial \mathcal{K}_j}{\partial x_i}(x)\right|=\sum_{2^j\leq |x|^{-1}}\left|\frac{\partial \mathcal{K}_j}{\partial x_i}(x)\right|
+\sum_{2^j> |x|^{-1}}\left|\frac{\partial \mathcal{K}_j}{\partial x_i}(x)\right|\leq c |x|^{-d-1}.$$
Therefore we obtain the following estimates
\begin{equation}\label{KK}
  |\mathcal{K}(x) | \leq c (1+|x|)^{-d} \qquad  \left|\frac{\partial \mathcal{K}}{\partial x_i}(x)\right| \leq c (1+|x|)^{-d-1}, \quad i=1,...,d.
\end{equation}
Note that  $\mathcal{K}$ is a $C^\infty$-function and all of its derivatives are bounded.
\par Now we come to the proof of (\ref{HC}).  Let  $r,s\in \mathbb{Z}^d$ with $|r|\geq 2|s|$. By mean value Theorem  we have
\begin{eqnarray*}
 |K(r-s)-K(r)|= |\mathcal{K}(r-s)-\mathcal{K}(r)|&\leq& |s|\int_{0}^1\sum_{i=1}^d\left|\frac{\partial \mathcal{K}}{\partial x_i}( r -ts)\right|dt.
\end{eqnarray*}
Using (\ref{KK}) and the fact that
$$|r -ts|\geq |r|-|s|\geq \frac{|r|}{2}$$
 we obtain the following
   \begin{eqnarray}\label{KKK}
 |K(r-s)-K(r)|\leq  \frac{c\;|s|}{(1+|r|)^{d+1}}.
 \end{eqnarray}
 Now use that
 $$\max_{1\leq i\leq d}(|r_i|)\leq|r|\leq d^{1/2}\max_{1\leq i\leq d}(|r_i|),$$
 we have
$$\sum_{|r|\geq 2|s|} \frac{1}{(1+|r|)^{d+1}}\leq c\; \sum_{\max_{1\leq i\leq d}(|r_i|)\geq 2d^{-1/2}|s|} \frac{1}{(1+\max_{1\leq i\leq d}(|r_i|) )^{d+1}}$$
It is not hard to see that
$$\sum_{r\in \mathbb{Z}^d,\; \max_{1\leq i\leq d}(|r_i|)=k} \leq c k^{d-1}.$$
and from which
\begin{equation}\label{KKKK}
 \sum_{|r|\geq  2d^{-1/2}|s|} \frac{1}{(1+|r|)^{d+1}}\leq c \sum_{k\geq  2d^{-1/2}|s|}^\infty\frac{1}{k^2}\leq \frac{c}{|s|}.
 \end{equation}
 Combine (\ref{KKKK}) with (\ref{KKK}),  yield that $$  \sum_{|r|\geq 2|s|}|K(r-s)-K(r)|   \leq c$$
which proves  (\ref{HC}).  The proof of Theorem \ref{th2} follows.
\end{proof}
In the next we shall be concerned with   Mikhlin  type  multiplier on $\mathbb{Z}$. Thus one can read Theorem \ref{th2} as follows
\begin{thm}\label{th11}
 If  $m$ is  a bounded $C^2-$function on $(0,1)$ such that
 \begin{equation}\label{M12}
 ( \xi  (1-\xi))^k| m^{(k)}(\xi)|\leq c;\qquad 0\leq k\leq 2
  \end{equation}
   then $T_m$ is a bounded operator from $\ell^p(\mathbb{Z})$ into $\ell^p(\mathbb{Z})$ for  $1<p<\infty$.
\end{thm}
 Clearly condition (\ref{M12}) is exactly  (\ref{mik})  when extending   $m $ to a periodic function.
 Now  we replace the interval $(0,1)$ by a bounded interval $(a,b)$.  For $f\in \ell^2(\mathbb{Z})$ we define its  Fourier transform by
$$\mathcal{F}_{\mathbb{Z}}^{a,b}(f)(\xi)=\sum_{n\in \mathbb{Z}}f(n)e^{i2\pi n\;\left(\frac{\xi-a}{b-a}\right)}=\mathcal{F}_{\mathbb{Z}} (f)\left(\frac{\xi-a}{b-a}\right)$$
 and its inverse
 $$(\mathcal{F}_{\mathbb{Z}}^{a,b})^{-1}(f)(n)=\frac{1}{b-a} \int_a^bf(\xi)e^{i2\pi n\;\left(\frac{\xi-a}{b-a}\right)}\;d\xi,\qquad n\in \mathbb{Z}.$$
For a   bounded function $m$ on  $(a,b)$ we define on $\ell^2(\mathbb{Z})$ the operator $T_m^{a,b}$ by
\begin{eqnarray*}
T_m^{a,b}(f)(n)&=&\frac{1}{b-a}\int_a^b m(\xi)\mathcal{F}_{\mathbb{Z}}^{a,b}(f)(\xi)e^{i2\pi n\;\left(\frac{\xi-a}{b-a}\right)}d\xi
\\&=&\int_0^1m(a+(b-a)\xi)\mathcal{F}_{\mathbb{Z}}^{a,b}(f)( a+(b-a)\xi)e^{i2\pi n}d\xi\\
&=&\int_0^1m(a+(b-a)\xi)\mathcal{F}_{\mathbb{Z}} (f)(\xi)e^{i2\pi n}d\xi,
\end{eqnarray*}
for $n\in \mathbb{Z}$. According to  Theorem \ref{th11} we  have  the following
\begin{cor}\label{cor1}
 If   $m$ is  a bounded $C^2$-function  on a bounded interval $(a,b)$,  such that for some constant $c>0$
  $$ (\xi-a)^k (b-\xi)^k |m^{(k)}(\xi)|\leq c,   \qquad0\leq k\leq2 $$
   then $T_m^{a,b}$ is a bounded operator from $\ell^p(\mathbb{Z})$ into $\ell^p(\mathbb{Z})$ for   $1<p<\infty$.
\end{cor}
As a typical example we have $m=\chi_{(a,b)}$  the characteristic function of $(a,b)$.
Theorem \ref{th11} can be generalized as follows.
\begin{thm}
  Let $ (a_j)_{0\leq j\leq s}$ be a  subdivision of   $[0,1]$. If $m$ is a  bounded $C^2$-function  on $(0,1)$ such that, for some constant $c>0$,
  $$ \left(\prod_{0\leq j\leq s}|\xi-a_j| ^k \right) |m^{(k)}(\xi)| \leq c; \qquad 0\leq k\leq 2. $$
   then $T_m$ is a bounded operator from $\ell^p(\mathbb{Z})$ into $\ell^p(\mathbb{Z})$ for    $1<p<\infty$.
\end{thm}
\begin{proof}
  Assume first that $s=2$ and let  $0=a_0 <a_1<a_2=1$. Let $\varepsilon> 0$ and $\varphi$ a be $C^\infty$ function on $\mathbb{R}$ such that
  $]a_1-2\varepsilon,a_1+2\varepsilon[\subset ]0,1[$ and  $\varphi(\xi)=0$ for
  $\xi \in]a_1- \varepsilon,a_1+ \varepsilon[$ and  $\varphi(\xi)=1$ for all
  $\xi \notin]a_1-2\varepsilon,a_1+2\varepsilon[$. Put
  $$m(\xi)= m(\xi)\varphi(\xi)+m(\xi)(1-\varphi(\xi))=m_1(\xi)+m_2(\xi),\qquad\xi\in(0,1).$$
 and
 $$T_m= T_{m_1}+T_{m_2}$$
 Clearly  the boundedness of $T_{m_1}$  is a consequence of Theorem \ref{th11}. However, if  we  consider   $\widetilde{m_2}$  the   $1$- periodic function such that $\widetilde{m_2}(\xi)=m_2(\xi)$ for $\xi\in [0,1]$,  then one can write
   \begin{eqnarray*}
  T_{m_2}(f)(n)&=&\int_{0}^{1}\widetilde{m_2}(\xi)\mathcal{F}_{\mathbb{Z}} (f)(\xi)e^{i2\pi n\xi}d\xi
  \\&=&\int_{a_1}^{a_1+1}\widetilde{m_2}(\xi)\mathcal{F}_{\mathbb{Z}} (f)(\xi)e^{i2\pi n\xi}d\xi
  \\&=&\int_{0}^{1}\widetilde{m_2}(\xi-a_1)\mathcal{F}_{\mathbb{Z}} (f)(\xi+a_1)e^{i2\pi n(\xi-a_1)}d\xi
  \\&=&e^{i2\pi na_1}\int_{0}^{1}\widetilde{m_2}(\xi+a_1)\mathcal{F}_{\mathbb{Z}} ( \widetilde{f})(\xi)e^{i2\pi n \xi}d\xi
  \\&=&e^{i2\pi na_1} T_{\widetilde{m_2}}^{a_1,a_1+1}(\widetilde{f})(n)
    \end{eqnarray*}
  where     $\widetilde{f}$ is the function  given  by $\widetilde{f}(n)=e^{i2\pi n a_1}f(n)$, $n\in \mathbb{Z}$. Now obseve that $\widetilde{m_2}$ satisfies the hypothesis of Corollary \ref{cor1} on $(a_1,a_1+1)$, then   the boundedness of $T_{m_2}$ follows.
\par Now for  $s\geq3$ we proceed as follows: choose $\varepsilon>0$ such that the intervals $]a_j-2\varepsilon,a_j+2\varepsilon[ $ are disjoint for all $1\leq j\leq s-1$
and  $C^\infty$ functions $\varphi_j$ with  $\varphi_j(\xi)=0$ for
  $\xi \in]a_j- \varepsilon,a_j+ \varepsilon[$ and  $\varphi(\xi)=1$ for
  $\xi \notin]a_j-2\varepsilon,a_j+2\varepsilon[$. Put
  $$\varphi= \frac{1}{s-1}\;\sum_{j=1}^{s-1}\varphi_j$$
and   write
$$m= m\varphi+m(1-\varphi)= m\varphi+\sum_{j=1}^{s-1}\frac{(1-\varphi_j)}{s-1}\;m= m_0+\sum_{j=1}^{s-1}m_j$$
and
$$T_m=T_{m_0}+\sum_{j=1}^{s-1}T_{m_j}.$$
Therefore from the above  argument  all the operators  $T_{m_j}$ are bounded on $\ell^p(\mathbb{Z})$, $1< p< \infty$.
\end{proof}
\section{ Applications}
\subsection{Discrete  Riesz Transforms}
  For a complex-valued function $f$ on $\mathbb{Z}^d$ its  discrete Laplacian is  given by
$$\Delta_d(f)(n)=\sum_{j=1}^d \partial_j\partial^*_jf(n)=\sum_{j=1}^ d \partial_j^*\partial_jf(n), \quad n\in \mathbb{Z}^d$$
where  $\partial_jf(n)=f(n+e_j)-f(n)$ and $ \partial^*_jf(n)=f(n)-f(n-e_j)$.  We have
$$\Delta_d(f)(n)=\sum_{j=1}^d  \Big(f(n+e_j)-2df(n)+f(n-e_j)\Big).$$
 The discrete Laplacian $\Delta_d$ is a bounded  self-adjoint operator on $\ell^2(\mathbb{Z}^d) $ and one has
  $$ \mathcal{F}_{\mathbb{Z}^d}(\Delta_d(f))(\xi)= -\sum_{j=1}^d |e^{i2\pi\xi_j}-1|^2\mathcal{F}_{\mathbb{Z}^d}(f)(\xi)=-4 \left(\sum_{j=1}^d \sin^2(\pi\xi_j)\right)\mathcal{F}_{\mathbb{Z}^d}(f)(\xi).$$
  The  discrete Riesz transforms $R_j$ , $j=1,...,d$, associated with $\Delta_d$ are defined     on $ \ell^2(\mathbb{Z}^d) $ as the multiplier operators
    $$ \mathcal{F}_{\mathbb{Z}^d} (R_j(f))(\xi)=\frac{e^{-i\pi\xi_j}\sin(\pi\xi_j) }{2\left(\sum_{k=1}^d \sin^2(\pi\xi_k)\right)^{1/2}}\mathcal{F}_{\mathbb{Z}^d} (f) (\xi),$$
   its can be    interpret   as $R_j(f)=\partial_j \Delta_d^{-1/2}$. Let us set
   $$\psi_j(\xi)=\frac{e^{-i\pi\xi_j}\sin(\pi\xi_j) }{2\left(\sum_{k=1}^d \sin^2(\pi\xi_k)\right)^{1/2}}$$
 and  prove that  $\psi_j$ satisfies the  Mikhlin condition (\ref{mik}). This can be seen by using the fact  that $\gamma\in\mathbb{N}^d$,
    $\partial^\gamma\psi_j $   is a linear combination  of  the  following functions
   $$ e^{-i\pi\xi_j}\prod_{k=1}^d\sin ^{\alpha_k}(\pi\xi_k)\prod_{k=1}^d\cos ^{\beta_k}(\pi\xi_k)   \left(\sum_{i=k}^d \sin^2(\pi\xi_k)\right)^{-\sum_{i=1}^d(\alpha_k+\beta_k)/2}$$
   where $\sum_{k=1}^d \alpha_k\leq |\gamma|+1$ and $\sum_{k=1}^d \beta_k\leq |\gamma|$. Hence  using the fact that for $k=1,...,d$, $\xi_k\in (-1/2,1/2)$ and $2|\xi_k|\leq |\sin\pi\xi_k|\leq \pi|\xi_k|$
we obtain
 $$  \left|\prod_{k=1}^d\sin ^{\alpha_k}(\pi\xi_i)\prod_{k=1}^d\cos ^{\beta_k}(\pi\xi_k)   \left(\sum_{k=1}^d \sin^2(\pi\xi_k)\right)^{-\sum_{k=1}^d(\alpha_k+\beta_k)/2}\right|\leq |\xi|^{ -\sum_{k=1}^d\beta_k} \leq |\xi|^{-|\gamma|}.$$
 Therefore we can apply Theorem \ref{th2}  to assert that $R_j$ is bounded on $\ell^p(\mathbb{Z}^d)$ for  $1<p<\infty$.
 \subsection{Imaginary powers of the discrete Laplace operator $\Delta_d$}
  Theorem \ref{th2} also applies to imaginary powers of the discrete Laplacian: $(-\Delta_d)^{it}$ for  $t\in \mathbb{R}$,
 it is the multiplier operator with multiplier  $\left(4\sum_{j=1}^d \sin^2(\pi\xi_j)\right)^{it}$
  \subsection{Strichartz type estimates  for discrete wave equation }
 We define the $d$-dimensional discrete wave equation    by
\begin{eqnarray}\label{wave}
&&\Delta_{ d} u(n,t)=\partial^2_t u(n,t),
\\&&u(n,0)=f(n),\quad \partial_t u(n,t)=g(n),  \quad (n,t)\in \mathbb{Z}^d\times \mathbb{R}\nonumber
\end{eqnarray}
where $f$ and $g$  are a given suitable  functions on $\mathbb{Z}^d$.   Considered as a  discrete counterpart of the continuous  wave equation,   many authors have been interested in studying this equation    see, for example, \cite{kop1, kop2, sla} and the references therein.
Putting
$$\phi(\xi)=2\sqrt{ \sum_{j=1}^d \sin^2(\pi\xi_j)}$$
and applying the discrete Fourier
transform, considering t as a parameter,  we deduce that  the solution of (\ref{wave})  can be written (formally) in the form
 $$u(n,t) =\mathcal{F}_{\mathbb{Z}^d}^{-1}\Big( \cos(t \phi(.))\mathcal{F}_{\mathbb{Z}^d}(f)\Big)(n)+ \mathcal{F}_{\mathbb{Z}^d}^{-1}\Big(\frac{\sin(t \phi(.))}{ \phi(.)} \mathcal{F}_{\mathbb{Z}^d}(f)\Big)(n).$$
We will prove the  following version of the Strichartz estimates.
 \begin{equation}\label{str}
  \|u(.,t)\|_{\ell^q }\leq c(t)\left(\|g\|_{\ell^p }+     \sum_{j=1}^d\|\partial_j f\|_{\ell^p }\right).
  \end{equation}
 for all $1<p\leq 2\leq q<\infty$.
  \par Let us obseve   first that   $\xi\rightarrow \sin(t \phi(\xi))/ \phi(\xi)$ is a   $C^\infty$- function  on $\mathbb{T}^d$  and  then in view of   Theorem  \ref{t11} we have
$$\left\|\mathcal{F}_{\mathbb{Z}^d}^{-1}\Big(\frac{\sin(t \phi(.))}{ \phi(.)} \mathcal{F}_{\mathbb{Z}^d}(g)\Big)\right\|_{\ell^q }
\leq c(t)\|g\|_{\ell^p }$$
whenever $1<p\leq q<\infty$.  To prove   (\ref{str})  it suffices to show   that
$$\left\|\mathcal{F}_{\mathbb{Z}^d}^{-1}\Big( \cos(t \phi(.))\mathcal{F}_{\mathbb{Z}^d}(f)\Big)\right\|_{\ell^q}
\leq c(t)    \sum_{j=1}^d\|\partial_j f\|_{\ell^p }.$$
 We write
$$ \cos( t\phi(\xi) )\mathcal{F}_{\mathbb{Z}^d}(f)(\xi)  =\frac{\cos t  \phi(\xi) }{  \phi(\xi)}\sum_{j=1}^d \mathcal{F}_{\mathbb{Z}^d}(R_j( \partial_jf))(\xi).$$
 As
$$\left|\frac{\cos( \phi(\xi))) }{  \phi(\xi)}\right|\leq \frac{c}{|\xi|}$$
it follows from  Theorem \ref{th1}  that
\begin{eqnarray*}
 \left\|\mathcal{F}_{\mathbb{Z}^d}^{-1}\Big( \cos(t \phi(.))\mathcal{F}_{\mathbb{Z}^d}(f)\Big)\right\|_{\ell^q }
&\leq& c(t)   \sum_{j=1}^d\|R_j(  \partial_j(f) \|_{\ell^p }
\end{eqnarray*}
and by using $\ell^p$-boundedness of $R_j$,
$$  \left\|\mathcal{F}_{\mathbb{Z}^d}^{-1}\Big( \cos(t \phi(.))\mathcal{F}_{\mathbb{Z}^d}(f)\Big)\right\|_{\ell^q }
\leq c(t)   \sum_{j=1}^d\|\partial_jf \|_{\ell^p }$$
   which conclude the proof of  (\ref{str}).


\begin{thebibliography}{HD}
\bibitem{CG}
O.  Ciaurri, T. Gillespie,  L. Roncal, J.L. Torrea, J. L. Varona,
\textit{Harmonic analysis associated with a discrete Laplacian}, Journal d'Analyse Math\'{e}matique,  132 (2017), 109-131.
\bibitem{CW1}
 R. R. Coifman and G. Weiss,\textit{ Analyse harmonique  non-commutatives sur certains
espaces homog\`{e}nes}, Lecture Notes in Math., vol.242, Springer-Verlag, Berlin
and New York, 1971.
\bibitem{CW1}
 R. R. Coifman and G. Weiss,\textit{ Extensions of Hardy spaces and their use in
analysis}, Bull. Amer. Math. Soc. 83 (1977), 56-645.
\bibitem{Hor}
L. H\"{o}rmander,\textit{ Estimates for translation invariant operators in $L^p$ spaces}, Acta Math. 104(1960), 93-139.
\bibitem{kop2}
 I. Egorova, E. Kopylova, and G. Teschl, \textit{Dispersion estimates for one-dimensional discrete
Schr\"{o}dinger and wave equations}, J. Spectr. Theory .

\bibitem{kop1}
E. Kopylova
\textit{On dispersion decay for discrete wave equations},Communications in Mathematical Analysis 17(2),209-216.
\bibitem{sla}
A. Slavik, \textit{Discrete-Space Systems of Partial Dynamic Equations and Discrete-Space Wave Equation}, Qualitative Theory of Dynamical Systems
16(2), 299-315.






\end{thebibliography}
 \end{document}